\documentclass[11pt]{article}

\usepackage{amsmath,amssymb,bm}
\usepackage{graphicx}
\usepackage{url}
\usepackage{cite}
\usepackage{theorem}        
\usepackage{enumerate}      

\usepackage[inner=30mm, outer=30mm, textheight=225mm]{geometry}

\def\C{\mathbb{C}}

\def\para{\mathcal{P}}
\def\rep{\mathcal{R}}

\def\PSL{PSL_2(\C)}

\def\tri{\mathcal{T}}

\newcommand{\co}{\colon\thinspace}

\newtheorem{theorem}{Theorem}[section]
\newtheorem{lemma}[theorem]{Lemma}
\newtheorem{proposition}[theorem]{Proposition}

\newtheorem{definition}[theorem]{Definition}

\newcommand{\prooflabel}{Proof}
\newcommand{\qedsymbol}{\rule[-0.5mm]{1.5mm}{3.0mm}}
\newcommand{\qed}{\nolinebreak\hfill\qedsymbol}
\newtheorem{proofthm}{\prooflabel}
\newenvironment{proof}{\begin{proofthm} \em}{\qed \end{proofthm}}

\begin{document}

\title{On the Hyperbolic Gluing Equations and Representations of Fundamental Groups of Closed $3$-Manifolds}
\author{Tian Yang}

\maketitle


\abstract{We show that for a representation of the fundamental group
of a triangulated closed 3-manifold (not necessarily hyperbolic)
into $\PSL$ so that any edge loop has non-trivial image under the
representation, there exist uncountably many solutions to the
hyperbolic gluing equation whose associated representations are
conjugate to the given representation, and whose volumes are equal
to the volume of the given representation.}

\section{Introduction}

In \cite{Th2}, Thurston introduced a system of algebraic equations
--called the \emph{hyperbolic gluing equations}¡ª for constructing
hyperbolic metrics on orientable $3$-manifolds with torus cusps. He
used solutions to the hyperbolic gluing equations to produce a
complete hyperbolic metric on the figure-eight knot complement in
the early stages of formulating his geometrization conjecture. On a
closed, oriented, triangulated $3$-manifold M, the
hyperbolic gluing equations can be defined in the same way: We
assign to each edge of each oriented tetrahedron in the
triangulation a shape parameter $z\in \mathbb{C}\setminus \{0,1\}$
such that

\begin{enumerate}[(a)]
\item opposite edges of each tetrahedron have the same shape parameter;
\item the three shape parameters assigned to three pairs of opposite
edges in each tetrahedron are $z,$ $\frac{1}{1-z}$ and
$1-\frac{1}{z}$ subject to an orientation convention; and
\item for each edge $e$ in M, if $z_1,...,z_k$ are
shape parameters assigned to the edges sharing $e$ as an edge, then we have
\begin{equation}\label{monodromy}
\prod_{i=1}^kz_i=1.
\end{equation}
\end{enumerate}

The equations (1) are termed \emph{the hyperbolic gluing equations}
, and the set of all solutions is the \emph{parameter space}
$\para(M).$ The space $\para(M)$ depends on the triangulation of
$M$. Given any $Z\in\para(M),$ the \emph{associated representation},
denoted $\rho_Z$, is defined by Yoshida in \cite{Yo}; and the
\emph{volume} of $Z$, denoted $Vol(Z)$, is well-defined using the
Lobachevsky-Milnor formula.
\\

In our joint work with F.Luo and S.Tillmann \cite{LTY}, we were able
to show the hyperbolic structure on a closed, oriented, triangulated
\emph{hyperbolic} $3$-manifold can be constructed from a solution to
the hyperbolic gluing equation using any triangulation with
essential edges. An edge in $\tri$ is termed \emph{essential} if it
is not a null homotopic loop in $M$. This is clear the case if it
has distinct end-point, but we allow the triangulation of $M$ to be
semi-simplicial (or singular), so that some or all edges may be
loops in $M$. It is well known that a closed $3$-manifold $M$ is
hyperbolic if and only if there exists a discrete and faithful
representation $\rho\co\pi_1(M)\rightarrow\PSL$ of the fundamental
group of $M$ into $\PSL$, and the main result of \cite{LTY} is to
construct a family of solutions to the hyperbolic gluing  equation
all of whose associated representations are discrete and faithful.
The proof makes a crucial use of Thurston's spinning construction
and a volume rigidity theorem attributed by Dunfield to Thurston,
Gromov and Goldman.
\\

Our main observation in the present paper is that the constrain that
the $3$-manifold $M$ is hyperbolic, or equivalently the
representation $\rho\co\pi_1(M)\rightarrow\PSL$ is discrete and
faithful, could be removed, and we have

\begin{theorem}\label{main}
Let $M$ be an oriented, closed $3$-manifold, $\tri$ be a
triangulation of $M$ so that each edge $e$ in $\tri$ is essential,
and $\rho\co\pi_1(M)\rightarrow\PSL$ be a representation of the
fundamental group of $M$ into $\PSL$ so that

\begin{equation}\label{T}
\rho([e])\neq 1,\ \text{for all loop}\ e\ \text{in}\ \tri.
\end{equation}

Then

\begin{enumerate}
\item there exist uncountably many solutions $Z_{\rho}$ to the hyperbolic gluing equation;

\item the associated representation
$\rho_{Z_{\rho}}$ is conjugate to $\rho$; and

\item $Vol(Z_{\rho})=Vol(\rho)$.
\end{enumerate}
\end{theorem}

We note that, in the special case that $\tri$ is simplicial, i.e.,
every $3$-simplex in $\tri$ has distinct vertices, every
representation $\rho\co\pi_1(M)\rightarrow\PSL$ satisfies condition
(\ref{T}); and in the case that $\rho\co\pi_1(M)\rightarrow\PSL$ is
discrete and faithful, Theorem \ref{main} implies Theorem 1.1
of \cite{LTY} as a special case.
\\

The present paper is organized as follows. In section 2, some basic
definitions on hyperbolic geometry are reviewed. Theorem \ref{main}
is proven in section 3 using the spinning
construction summarized in \cite{LTY} and a theorem of Luo on the
continuous extension of the volume function \cite{Lu1}.

\section{The parameter space}

\subsection{The hyperbolic gluing equation and the volume of solutions}

If $\sigma$ is an oriented 3-simplex with edges from one vertex
labeled by $e_1,$ $e_2$ and $e_3$ so that the opposite edges have
the same labeling, then the cyclic order of $e_1,$ $e_2$ and $e_3$
viewed from each vertex depends only on the orientation of the
tetrahedron, i.e. is independent of the choice of the vertices. Note
that each pair of opposite edges $e_i$ corresponds to a normal
isotopy class of quadrilateral (\textit{normal quadrilateral} for
short) $q_i$ in $\sigma$ so that $q_1\rightarrow q_2\rightarrow
q_3\rightarrow q_1$ is the cyclic order induced by the cyclic order
on the edges from a vertex. To define hyperbolic gluing equation, we
need the following notation. Let $e$ be an edge in $\tri,$ and $q$ be a
normal quadrilateral in $\sigma.$ The index $i(q,e)$ is the integer
$0,$ $1$ or $2$ defined as follows. $i(q,e)=0$ if $e$ is not an edge
of $\sigma.$ $i(q,e)=1$ if $e$ is the only edge in $\sigma$ facing
$q$ and $i(q,e)=2$ if $e$ are the two edges in $\sigma$ facing $q.$
Let $Q$ be the set of normal quadrilaterals in $\tri,$ we have

\begin{definition}\label{ThGE}
Suppose $(M,\tri)$ is a triangulated oriented close $3$-manifold.
Thy hyperbolic gluing equation is defined for
$Z=(z_q)\in(\mathbb{C}\setminus\{0,1\})^Q$ so that
\begin{enumerate}[(a)]
\item for each edge $e$ in $\tri,$

$$\prod_{q\in Q} z_q^{i(q,e)}=1,$$

and
\item if $\sigma$ is a $3$-simplex in $\tri,$ and $q_1\rightarrow
q_2\rightarrow q_3\rightarrow q_1$ is the cyclic order of normal
quadrilaterals in $\sigma,$ then

$$z_{q_{i+1}}=\frac{1}{1-z_{q_i}},$$

where $q_{3+1}$ is understood to be $q_1.$
\end{enumerate}

The set of all solutions to the hyperbolic gluing equation is call
the parameter space, and is denoted by $\para(M)$.
\end{definition}
Let $z_{\sigma}=(z_{q_1},z_{q_2},z_{q_3})$ be the complex numbers
assigned to $q_i,$ $i\in\{1,2,3\},$ then we have

\begin{definition}\label{vol}
The volume of $z_{\sigma}$ is defined to be the sum of the
Lobachevsky functions

\begin{equation*}
\begin{split}
Vol(z_{\sigma})=&\sum_{i=1}^3\Lambda(\arg(z_{q_i}))\\
\doteq&\sum_{i=1}^3(-\int_0^{\arg(z_{q_i})}\ln
|2\sin t|dt);\\
\end{split}
\end{equation*}

and the volume of $Z=(z_{\sigma})\in\para(M)$ is defined by

$$Vol(Z)=\sum_{\sigma\in T}Vol(z_{\sigma}).$$
\end{definition}

\subsection{The shape parameters of an ideal tetrahedron}
Let $\overline{\mathbb{H}^3}=\mathbb{H}^3\cup S_{\infty}^2$ be the compactification of $\mathbb{H}^3$, where $S_{\infty}^2$ is the \emph{sphere at infinity}. We have

\begin{definition}\label{shape}
Let $\sigma$ be an ideal tetrahedron in $\overline{\mathbb{H}^3}$
with vertices $\{v_i\}\subset S_{\infty}^2,$ $i\in\{1,...,4\},$ and $e_{ij}$ be the edge from $v_i$ to $v_j$. Identifying
$S_{\infty}^2$ with $\mathbb{C}\cup\{\infty\},$ the shape parameter
of $\sigma$ at $e_{ij}$ is defined by the following cross-ratio

\begin{equation*}
\begin{split}
z_{ij}\doteq& (v_i,v_j;v_k,v_l)\\
=&\frac{v_i-v_k}{v_i-v_l}\cdot\frac{v_j-v_l}{v_j-v_k}\\
\end{split}
\end{equation*}

where $(i,j,k,l)$ is an even permutation of $(1,2,3,4)$ .
\end{definition}

A direct cross-ratio calculation shows the following well known

\begin{proposition}\label{wellknown}

\begin{enumerate}
\item For all $\{i,j\}, \{k,l\}\subset\{1,...,4\},$ $i,j\neq k,l,$

$$z_{ij}=z_{kl},\ \,$$

i.e., opposite edges share the same shape parameter, so we can
denote the shape parameter of $\sigma$ at $e_{ij}$ and $e_{kl}$ by
$z_q,$ where $q$ is the normal quadrilateral facing $e_{ij}$ and
$e_{kl},$ and
\item if $q_1\rightarrow q_2\rightarrow q_3\rightarrow q_1$ is the
cyclic order of normal quadrilaterals in $\sigma,$ then

$$z_{q_{i+1}}=\frac{1}{1-z_{q_i}},$$

where $q_{3+1}$ is understood to be $q_1.$
\end{enumerate}
\end{proposition}

For an ideal tetrahedron $\sigma$ with shape parameters
$z_{q_1},z_{q_2}$ and $z_{q_3},$ the hyperbolic volume is calculated
by Milnor as

\begin{equation*}\label{Milnor}
Vol_{\mathbb{H}^3}(\sigma)=\sum_{i=1}^3\Lambda(\arg(z_{q_i})).
\end{equation*}

We call an ideal tetrahedron $\sigma\subset\overline{\mathbb{H}^3}$ \emph{flat} if it lies in a totally geodesic plan. When $\sigma$ is flat, we have that $\{z_{q_i}\}$ are real numbers and
$Vol(\sigma)=0$.

\subsection{The associated representation}\label{associated}

Given a solution $Z\in \para(M)$ to the hyperbolic gluing equation
for a triangulated cusped $3$-manifold $(M,\tri)$ with essential
edges, the associated representation
$\rho_Z\co\pi_1(M)\rightarrow\PSL$ was described by Yoshida
\cite{Yo} via constructing the pseudo developing maps. For a closed
triangulated $3$-manifold with essential edges, the construction of
the associated representation is essentially the same as Yoshida's.
Namely, for each $Z\in\para(M)$, there is a continuous map
$D_Z\co\widetilde{M}\rightarrow\overline{\mathbb{H}^3}$ taking
$3$-simplices in $\widetilde{\tri}$ to ideal straight simplices, and
a representation $\rho_Z$ which makes it invariant. This
construction is described in details in Section 4.5, \cite{LTY}.

\section{The existence of solutions to the hyperbolic equation}

\subsection{The proof of 1. of Theorem \ref{main}}\label{existence}
\begin{proof}
Let $\pi\co\widetilde{M}\rightarrow M$ be the universal cover of
$M$, and $\widetilde{\tri}$ the triangulation of $\widetilde{M}$
induced from $\tri$. Let $V$ be the set of vertices of $\tri$.

We take any $\rho$-equivariant map $F\co \pi^{-1}(V)\rightarrow
S_{\infty}^2$ so that for any $3$-simplex $\sigma$ in $\tri$, the
four points $\{F(v)\ |\ v\ \text{is a vertex of}\ \sigma\}$ are
distinct. The existence of such $F$ is guaranteed by condition
(\ref{T}). Indeed, let $D\subset\widetilde{M}$ be a fundamental
domain of $M$ which is a union of $3$-simplices of
$\widetilde{\tri}$. Let $V=\{v_1,...,v_{|V|}\}$ and
$V_i=\pi^{-1}(v_i)\cap D$. We take a point $u_i\in V_i$ for each
$i\in\{1,...,|V|\}$. Then by the $\rho$-equivariance, $F(V_i)$
should be determined by $F(u_i)$. Namely, if $u_i'\in V_i$ with
$u_i'=\gamma\cdot u_i$ for some $\gamma\in\pi_1(M)$, then
$F(u_i')=\rho(\gamma)F(u_i).$ Let $e$ be an edge in $D$ such that
both of its vertices $w_1$, $w_2$ are in $V_i$ for some
$i\in\{1,...,|V|\}$, i.e., $w_j=\gamma_j\cdot u_i$ for some
$\gamma_j\in\pi_1(M)$, $j\in\{1,2\}$, we see that since by condition
(\ref{T}), $\rho([e])\neq1$, there are at most two $y\in
S_{\infty}^2$ such that $\rho([e])\cdot y=y.$ Therefore, for a
generic choice of $z\in S_{\infty}^2$,

$$\rho(\gamma_2)\cdot z=\rho([e])\cdot(\rho(\gamma_1)\cdot z))\neq\rho(\gamma_1)\cdot z.$$

Since there are in total finitely many edges in $D$,
for a generic choice of $(z_1,...,z_{|V|})\in(S_{\infty}^2)^V$, the
map defined by

$$F(u_i)\doteq z_i,\ \ i\in\{1,...,|V|\}$$

and

$$F(\gamma\cdot u_i)\doteq\rho(\gamma)\cdot F(u_i),\ \
\gamma\in\pi_1(M)$$

satisfies the property that for each edge $e$ of $D$ with vertices
$w_i$ and $w_2$ such that $w_1,w_2\in V_i$ for some $i\in\{1,...,|V|\}$,

$$F(w_2)=\rho([e])\cdot
F(w_1)\neq F(w_1).$$

Furthermore, since there are only finitely many vertices in $D$, for a generic choice of $(z_1,...,z_|V|)\in(S_{\infty}^2)^V$, the map $F$ satisfies that for each $e$ in $D$ with vertices $w_1$ and $w_2$,

$$F(w_2)\neq F(w_1).$$

Therefore, for each $3$-simpliex $\sigma$ in $\tri$, the four points
$\{F(v)\ |\ v\ \text{is a vertex of }\sigma\}$ are distinct.

For any $3$-simplex $\sigma$ of $\tri$, let $\widetilde{\sigma}$ in
$\widetilde{\tri}$ be a lift of $\sigma$. Then the four distinct
points $\{F(v)\ |\ v\ \text{is a vertex of}\ \sigma\}$ determine an
ideal tetrahedron $\sigma_{\infty}$ with them as vertices. We assign
the shape parameters of $\sigma_{\infty}$ to the corresponding
normal quadrilaterals of $\sigma$, and get an assignment
$Z_{\rho}\subset(\mathbb{C}\setminus\{0,1\})^Q$ of complex numbers
to the set of normal quadrilaterals in $\tri$. We claim that
$Z_{\rho}$ is a solution to the hyperbolic gluing equation.

By Proposition \ref{wellknown}, we see that $Z_{\rho}$ satisfies (b)
of Definition \ref{ThGE}. The verification of (a) is a cross-ratio
calculation which is exactly the same as in \cite{LTY}. We include
it here for the readers' convenience. Let $e\in E$, and
$\widetilde{e}\in\widetilde{\tri}$ a lift of $e$ with end points
$v$ and $w$. Let $\sigma_1,...,\sigma_k$ be tetrahedra in
$\widetilde{\tri}$ sharing $\widetilde{e}$ as an edge in a cyclic
order, and $q_i\subset\sigma_i$ be the normal quadrilateral facing
$\widetilde{e}$. Let $u_i$ and $u_{i+1}$ be the other two vertices
of $\sigma_i$ so that $u_i\in\sigma_{i-1}\cap\sigma_i$. We make a
convention that $u_{k+1}=u_1$. Let $\sigma_{i,\infty}$ be the ideal
tetrahedron determined by vertices $F(v),F(w),F(u_i)$ and
$F(u_{i+1})$, and $l$ be the geodesic connecting $F(v)$ and $F(w)$,
then $\{\sigma_{i,\infty}\}_{i=1}^k$ share $l$ as an edge in a
cyclic order. Without loss of generality, we can assume that
$F(v)=0$ and $F(w)=\infty\in S_{\infty}^2=\mathbb{C}\cup\{\infty\}$.
Suppose $z_i$ is the complex number assigned to $q_i$, i.e., the
shape parameter of $\sigma_{i,\infty}$ at $l$, then

\begin{equation*}
\begin{split}
\prod_{q\in Q}z_q^{i(q,e)}=&\prod_{i=1}^kz_i\\
=&\prod_{i=1}^k(0,\infty;F(u_i),F(u_{i+1}))\\
=&\prod_{i=1}^k\frac{F(u_i)}{F(u_{i+1})}\\
=&1,
\end{split}
\end{equation*}

which verifies (a).
\\

From the arbitrariness of $F$, we see that there are uncountably
many choices of $Z_{\rho}$.
\end{proof}

We call the solutions $Z_{\rho}\in\para(M)$ constructed above \emph{the solutions from the spinning construction}.

After we obtained the result, it was brought to our attention that
similar construction had also appeared a little earlier in the work
of Kashaeve, Korepanov and Martyushev \cite{KKM}.

\subsection{Spinning construction and the proof of 2. of Theorem
\ref{main}}\label{spinning}

According to Thurston's notes, Section 6.1, \cite{Th2}, any $k+1$
points $v_0,...,v_k$, $1\leqslant k\leqslant 3$, in $\mathbb{H}^3$
determine a \emph{straightening map} (or \emph{straight
$k$-simplex})
$\sigma_{v_0,...,v_k}\co\Delta^k\rightarrow\mathbb{H}^3$, whose
image is the convex hull of $v_0,...,v_k$. Similarly, any $k+1$
points $v_0,...,v_k$, $1\leqslant k\leqslant 3$, in $S_{\infty}^2$
determine an \emph{ideal straightening map} (or \emph{ideal straight
$k$-simplex}), see Section 2.2, \cite{LTY} for details. The (ideal)
straightening map is natural in the following sense (see also
\cite{LTY} for the proof).

\begin{proposition}\label{straight}
\begin{enumerate}
\item If $\Delta'$ is an $m$-face of $\Delta^k$ so that
$\sigma_{v_0,...,v_k}(\Delta')$ has vertices $v_{i_0},...,v_{i_m},$
then
$$\sigma_{v_0,...,v_k}|_{\Delta'}=\sigma_{v_{i_0},...,v_{i_m}}.$$

\item If $g\in Iso(\mathbb{H}^n),$ the group of isometries of
$\mathbb{H}^3,$ then
$$g\circ\sigma_{v_0,...,v_k}=\sigma_{g\cdot v_0,...,g\cdot v_k}.$$
\end{enumerate}
\end{proposition}

To prove 2. of Theorem \ref{main}, we need the following technical
Lemma whose proof is contained in \cite{LTY}.

\begin{lemma}\label{tech}
Let $\{\sigma_t\co\Delta^k\rightarrow\mathbb{H}^3\ \
t\in\mathbb{R}_{\geqslant0}\}$ be a family of straight $k$-simplices
so that the $i$-th vertex $v_{i,t}$ of $\sigma_t$ lies in a geodesic
ray $l_i$ and $v_{i,t}$ moves toward the end point $v_l^*$ of $l_i$
at unit speed, i.e., $ d(v_{i,0},v_{i,t})=t.$ If $v_0^*,...,v_k^*$
are pairwise distinct, then as $t$ tends to $\infty$ the family
$\{\sigma_t\}$ converges pointwise to an ideal straight $k$-simplex
$\sigma_{\infty}\co\Delta^k\rightarrow\overline{\mathbb{H}^3}$ whose
vertices are $v_0^*,...,v_k^*.$
\end{lemma}

\begin{proof}\textbf{\hspace{-10pt} of 2. of Theorem \ref{main}}

We take an arbitrary $\rho$-equivariant map
$f\co\widetilde{M}\rightarrow\mathbb{H}^3$ and apply the following
spinning construction. Let $D\subset\widetilde{M}$ be a fundamental
domain of $M$ which is a union of $3$-simplices of
$\widetilde{\tri}$ with some $0$-, $1$- and $2$-faces removed so
that $\pi|_D\co D\rightarrow M$ is one-to-one and onto, and $\tri_D$
be the triangulation of $D$ restricted from $\widetilde{\tri}$. Let
$V_D$ be the set of vertices of $\tri_D$. For each $v\in V_D$, let
$l_v$ be a geodesic in $\mathbb{H}^3$ passing through $f(v)$
$F(v)\in S^2_{\infty}$ as one of its end-points. We parameterize
$l_v\co(-\infty,\infty)\rightarrow\mathbb{H}^3$ so that
$l_v(0)=f(v)$, $\|l_v'(t)\|_{\mathbb{H}^3}=1,\ \forall
t\in(-\infty,\infty)$, and $l_v(t)\rightarrow F(v)$ as $t\rightarrow
+\infty$, and define a family of piecewise smooth $\rho$-equivariant
maps $f_t\co\widetilde{M}\rightarrow\mathbb{H}^3$, $t\in[0,\infty)$
as follows. Define

$$f_t(v)=exp_v(t\cdot l_v'(0)),\ \forall v\in V_D,$$ and

$$f_t(\gamma\cdot v)=\rho(\gamma)\cdot f_t(v),\ \forall \gamma\in\pi_1(M),\ v\in V_D.$$

Extend $f_t$ to the $1$-, $2$- and $3$-simplices of
$\widetilde{\tri}$ by straightening maps. By 1. of Proposition
\ref{straight}, $f_t$ is well defined, and by 2. of Proposition
\ref{straight}, $f_t$ is $\rho$-equivariant. From the definition of
$f_t$, we see that for each vertex $\widetilde{v}$ of $\widetilde{\tri}$, $f_t(\widetilde{v})$ approaches to $F(\widetilde{v})\in S_{\infty}^2$, and for each face $\widetilde{\sigma}$ of $\widetilde{\tri}$, $f_t(\widetilde{\tri})$ lies in a totally geodesic plane.

By Lemma \ref{tech}, $f_t\co\widetilde{M}\rightarrow\mathbb{H}^3$ pointwise converges to a piecewise smooth $\rho$-equivariant map
$f_{\infty}\co\widetilde{M}\rightarrow\overline{\mathbb{H}^3}$ such
that
\begin{enumerate}
\item $\forall \widetilde{v}\in \pi^{-1}(V)$, $f_{\infty}(\widetilde{v})=F(\widetilde{v})$; and

\item $f_{\infty}(\widetilde{M}\setminus\pi^{-1}(V))\subset\mathbb{H}^3$.
\end{enumerate}

Given the solution $Z_{\rho}$ to the hyperbolic gluing equation, $f_{\infty}$ can be regarded as the pseudo developing map described in Section \ref{associated} which gives rise to the associated representation. Tautologically the map
$f_{\infty}\co\widetilde{M}\setminus\pi^{-1}(V)\rightarrow\mathbb{H}^3$
is $\rho_{Z_{\rho}}$-equivariant. Therefore, $f_{\infty}$ is both $\rho$- and
$\rho_{Z_{\rho}}$-equivariant, and for all $\gamma\in\pi_1(M)$ and
$x\in\widetilde{M}\setminus\pi^{-1}(V)$; and we have

\begin{equation*}
\begin{split}
\rho_{Z_{\rho}}(\gamma)\cdot f_{\infty}(x)=&f_{\infty}(\gamma\cdot x)\\
=&\rho(\gamma)\cdot f_{\infty}(x),\\
\end{split}
\end{equation*}

i.e.,
$\rho_{Z_{\rho}}(\gamma)|_{f_{\infty}(\widetilde{M}\setminus\pi^{-1}(V))}=\rho(\gamma)|_{f_{\infty}(\widetilde{M}\setminus\pi^{-1}(V))}$.
It is clear that $f_{\infty}(\widetilde{M}\setminus\pi^{-1}(V))$
contains more than four points. Indeed, for and $3$-simplex
$\widetilde{\sigma}$ in $\widetilde{\tri}$, the interior of the
ideal tetrahedron $f_{\infty}(\widetilde{\sigma})$ contains an open
subset of a totally geodesic plane (generically, the interior of the
ideal tetrahedron $f_{\infty}(\widetilde{\sigma})$ is itself open in
$\mathbb{H}^3$, and the only ``bad'' extremal case happens only if
that for any $\widetilde{\sigma}$ in $\widetilde{\tri}$,
$f_{\infty}(\widetilde{\sigma})$ is flat). Therefore,
$$\rho_{Z_{\rho}}(\gamma)=\rho(\gamma)\in\PSL,\ \forall\gamma\in\pi_1(M),$$ i.e.,
$$\rho_{Z_{\rho}}=\rho\co\pi_1(M)\rightarrow\PSL.$$
\end{proof}

Let $\rep(M)$ be the set of representations $\rho\co\pi_1(M)\rightarrow\PSL$. As we pointed out in the introduction, in the case that the triangulation $\tri$ is simplicial, every $\rho\in\rep(M)$ satisfies condition (\ref{T}), and we have

\begin{theorem}
If $M$ is a closed, oriented $3$-manifold, and $\tri$ is a simplicial triangulation of $M$, then the map $Y\co\para(M)\rightarrow\rep(M)$ defined by $Y(Z)=\rho_Z$ is surjective.
\end{theorem}

By 2. of Theorem \ref{main}, we see that our construction is the inverse construction of Yoshida's as described in Section \ref{associated}, and we have

\begin{theorem}
If $(M,\tri,\rho)$ satisfies the condition of Theorem \ref{main},
then all the solutions $Z_{\rho}$ of the hyperbolic gluing equation
such that $\rho_{Z_{\rho}}$ is conjugate to $\rho$ are from the
spinning construction.
\end{theorem}

\subsection{Continuous extension of the volume function and the proof of 3. of Theorem \ref{main}}

Given a hyperbolic $3$-simplex $\sigma$ with vertices $v_1,...,v_4$,
the $i$-th face is defined to be the $2$-simplex facing $v_i$. The
dihedral angle between the $i$-th and $j$-th faces is denoted by
$a_{ij}(\sigma)$. As a convention, we define $a_{ii}(\sigma)=\pi$,
and call the symmetric matrix $[a_{ij}(\sigma)]_{6\times6}$ the
angle matrix of $\sigma$. It is well known that the angle matrix
$[a_{ij}(\sigma)]_{6\times6}$ determines $\sigma$ up to isometry.

To prove 3. of Theorem \ref{main}, we need the following theorem of
Luo \cite{Lu1}.

\begin{theorem}\textbf{(Luo)}\label{Luo}
Let $X\subset\mathbb{R}^{6\times6}$ be the space of angle matrices
of all hyperbolic $3$-simplices. The volume function $V\co
X\rightarrow \mathbb{R}$ can be extended continuously to the closure
of $X$ in $\mathbb{R}^{6\times6}$.
\end{theorem}

We point out that Luo's original result, Theorem 1.1 in \cite{Lu1},
is more general than Theorem \ref{Luo}. It covers the cases of
simplices in arbitrary dimensions, and in both hyperbolic and
spherical geometry. See \cite{Lu1} for details.

\begin{proof}\textbf{\hspace{-10pt} of 3. of Theorem \ref{main}}

Let $Z_{\rho}$ be a solution to the hyperbolic gluing equation
constructed in Section \ref{existence}. By Definition \ref{vol},

\begin{equation*}
\begin{split}
 Vol(Z_{\rho})=&\sum_{\sigma\in T}Vol(z_{\rho,\sigma})\\
 =&\sum_{\sigma\in \tri_D}Vol_{\mathbb{H}^3}(f_{\infty}(\sigma))\\
\end{split}
\end{equation*}

Since $f_t\co\widetilde{M}\rightarrow\mathbb{H}^3$ constructed in
Section \ref{spinning} is $\rho$-equivariant, $\forall t\in
[0,+\infty)$, as defined by Dunfield in Section 2.5 of \cite{Du},

\begin{equation*}
\begin{split}
Vol(\rho)=&\int_Df_t^*(dVol_{\mathbb{H}^3})\\
=&\sum_{\sigma\in\tri_D}\int_{\sigma}f_t^*(dVol_{\mathbb{H}^3})\\
=&\sum_{\sigma\in\tri_D}Vol_{\mathbb{H}^3}(f_t(\sigma)),\ \ \forall
t\in [0,+\infty).\\
\end{split}
\end{equation*}

For any $3$-simplex $\sigma\in\tri_D$, since $f_t|_{\sigma}$
pointwise converges to $f_{\infty}|_{\sigma}$, the angle matrices
$[a_{ij}(f_t(\sigma))]_{6\times6}$ converges to
$[a_{ij}(f_{\infty}(\sigma))]_{6\times6}\in\overline{X}$. By Theorem
\ref{Luo},

\begin{equation*}
Vol_{\mathbb{H}^3}(f_{\infty}(\sigma))=\lim_{t\rightarrow
+\infty}Vol_{\mathbb{H}^3}(f_t(\sigma)).
\end{equation*}

Therefore,
\begin{equation*}
\begin{split}
 Vol(Z_{\rho})=&\sum_{\sigma\in \tri_D}Vol_{\mathbb{H}^3}(f_{\infty}(\sigma))\\
=&\sum_{\sigma\in \tri_D}\lim_{t\rightarrow+\infty}Vol_{\mathbb{H}^3}(f_t(\sigma))\\
=&\lim_{t\rightarrow+\infty}\sum_{\sigma\in \tri_D}Vol_{\mathbb{H}^3}(f_t(\sigma))\\
=&\lim_{t\rightarrow +\infty}Vol(\rho)\\
=&Vol(\rho)\\
\end{split}
\end{equation*}
\end{proof}

%
%
%
\section*{Acknowledgment}

Research of the author is partially supported by the NSF. The author
would like to thank Feng Luo and Stephan Tillmann for useful
discussions, Feng Luo for encouraging him to write up this result,
and Stephan Tillmann for pointing out an error in an earlier draft,
sharing important ideas and providing instructive suggestions
generously.

\bigskip
\noindent
Tian Yang\\
Department of Mathematics, Rutgers University\\
New Brunswick, NJ 08854, USA\\
(tianyang@math.rutgers.edu)

\end{document}